\documentclass[a4paper, 10pt]{article}
\usepackage[utf8]{inputenc}
\usepackage[T1]{fontenc}
\usepackage{lmodern}
\usepackage[left=3cm,right=3cm,top=2cm,bottom=2cm]{geometry}
\usepackage{amsmath, amssymb}
\usepackage{amsthm}
\usepackage{graphicx}
\usepackage[dvipsnames]{xcolor}
\usepackage{pgf,tikz,tikz-cd}
\usetikzlibrary{arrows}
\usepackage{array}
\usepackage{verbatim}
\usepackage{multicol}
\usepackage{mdframed}
\usepackage{pstricks}
\usepackage{pstricks-add}
\usepackage{tabularx}
\usepackage{fancybox}
\usepackage{stmaryrd}
\usepackage{comment} 
\usepackage{caption}
\usepackage{mathrsfs}
\usepackage{enumitem}
\usepackage[all, cmtip]{xy}
\usepackage{pdfpages}
\usepackage{titlesec}
\usepackage{supertabular}
\usepackage{faktor}
\usepackage{todonotes}
\usepackage{quiver}

\usepackage{hyperref}
\hypersetup{
    colorlinks=true,
    linkcolor=BlueViolet,
    filecolor=magenta,      
    urlcolor=violet,
    pdftitle={},
    pdfpagemode=FullScreen,
    }

\theoremstyle{remark}
\newtheorem*{notation}{Notation}
\newtheorem{theo}{test}[section]
\newtheorem{rmk}[theo]{Remark}
\newtheorem{ex}[theo]{Example}
\theoremstyle{plain}

\theoremstyle{plain}
\usepackage{thmbox}

\newtheorem{theorem}[theo]{Theorem}
\newtheorem{prop}[theo]{Proposition}
\newtheorem{lemma}[theo]{Lemma}
\newtheorem{cor}[theo]{Corollary}
\newtheorem{definition}[theo]{Definition}

\DeclareMathOperator{\Int}{Int}
\DeclareMathOperator{\Conv}{Conv}
\DeclareMathOperator{\ord}{ord}

\newcommand{\Ocal}{\mathcal{O}}

\newcommand{\Zcal}{\mathcal{Z}}

\newcommand{\inj}{\hookrightarrow}

\newcommand{\Ncal}{\mathcal{N}}

\newcommand{\N}{\mathbb{N}}
\newcommand{\Z}{\mathbb{Z}}
\newcommand{\Q}{\mathbb{Q}}
\newcommand{\R}{\mathbb{R}}
\newcommand{\C}{\mathbb{C}}

\date{\today}
\author{François Bernard, Enrique Chávez-Martínez, Arturo E. Giles Flores}
\title{Lipschitz saturation of toric singularities in any dimension}
\begin{document}

\maketitle

\begin{abstract}
    We describe the semigroup of the Lipschitz saturation of a complex analytic toric singularity in arbitrary dimension. We give a necessary and sufficient condition for a monomial in the normalization to belong to the Lipschitz saturation, in terms of Newton polyhedra and lattice conditions, and deduce a finite algorithm to compute it. We also show that, in dimension $\geq 2$, Campillo’s notion of presaturation differs from the Lipschitz saturation, even for complex singularities.
\end{abstract}

\section*{Introduction}

 The study of Lipschitz geometry of singular spaces has undergone significant development in recent years, revealing deep connections between analytic, algebraic, and metric properties of singularities. The study of Lipschitz saturation of complex analytic varieties originates from the work of F. Pham and B. Teissier in the late 1960s \cite{PhamTeissier2020LipschitzFractions} and was inspired by Zariski’s theory of saturation, whose objective was to establish the foundations for an algebraic theory of equisingularity \cite{Zariski1968EquisingularityIII,Zariski1971SaturationI,Zariski1971SaturationII,Zariski1975SaturationIII}.\\

For a reduced analytic algebra $A$, the normalization $\overline{A}$ corresponds to the ring of germs of locally bounded meromorphic functions. The Lipschitz saturation $A^s$ is an intermediate ring formed by the germs of locally bounded Lipschitz meromorphic functions, and these can be defined via the integral closure of a certain ideal. In this setting, the ring $A^s$ is again an analytic algebra which in some cases coincides with Zariski's algebraic saturation. \cite{PhamTeissier2020LipschitzFractions, FGST2020LipschitzAnAlgebraicApproch}.This algebraic description opened the way to define the relative Lipschitz saturation of rings in a more general context \cite{Lipman1975}, and we should say  it has recently found renewed interest as shown by the papers \cite{ Bernard2025RelativeLipSat,DaSilvaRibeiro2023,DaSilvaNetto2024SurveyLipSat}.

In the case of analytic algebras, beyond the algebraic relevance of the construction, the associated germ $(X^s,0)$ is important in the study of bi-Lipschitz equisingularity. While this construction is well understood in the case of complex curve singularities (1-dimensional reduced analytic algebras \cite{BrianconGallicoGranger1980DeformationsEqu,FGST2020LipschitzAnAlgebraicApproch}), where it admits an explicit and algorithmic description in terms of semigroups, its higher-dimensional counterpart remains largely unexplored and poorly understood. Recently, in \cite{DuarteGiles2023ToricLipSat} the authors prove that for a toric singularity with smooth normalization, the Lipschitz saturation is again toric, but no general description of its semigroup is given.

The goal of this paper is to provide such a description. More precisely, we give an explicit and computable characterization of the semigroup associated with the Lipschitz saturation of a toric singularity with smooth normalization. Our approach combines several key ingredients: a reinterpretation of Campillo’s construction \cite{Campillo1983OnSatOfCurveSing}, a detailed analysis of arcs via the valuative criterion, and a systematic use of Newton polyhedra and support functions.\\

An important role is played by what we call Campillo’s criterion, which provides a practical method to construct elements in the Lipschitz saturation by adjoining specific fractions from the normalization. While this criterion is known to be sufficient in general and complete in dimension one, we show how it can be effectively exploited in higher dimensions when combined with convex-geometric techniques. In section \ref{SecCampillo} we recall this criterion and translate it to the toric case via the language of the associated semigroups.
Our main result is Theorem \ref{Th_Main} which gives a necessary and sufficient condition for a monomial to belong to the Lipschitz saturation in terms of its position with respect to certain Newton polyhedra and lattice conditions. As such, section \ref{Section_CriterionForNotinLipSat} is dedicated to develop a criterion that determines which monomials are not in the Lipschitz saturation. In section \ref{SecMain}, combining Campillo's criterion and a generalization of the techniques used in \cite{BrianconGallicoGranger1980DeformationsEqu} to study the 1-dimensional case we obtain a characterization of the monomials which do belong to the Lipschitz saturation.  This leads to a finite and constructive procedure to determine the saturated semigroup.
In addition, we exhibit examples showing that, in contrast with the curve case, Campillo’s saturation does not coincide with the Lipschitz saturation in higher dimension. This highlights new phenomena specific to the multidimensional setting and highlights the subtlety of Lipschitz conditions beyond the one-dimensional case. \\ 

On the occasion of his 80th birthday, we dedicate this work to Bernard Teissier with deep admiration and affection. By laying the algebraic foundations of Lipschitz geometry, he charted the very paths we explore here today. We hope this work reflects a small fraction of the inspiration, rigor, and joy he has brought to singularity theory. 

\paragraph{Acknowledgement}
The first author was partially supported by Plan d’investissements
France 2030, IDEX UP ANR-18-IDEX-000. The second and third-named authors were supported by SECIHTI project CF-2023-G-33.

\section{Campillo's criterion}\label{SecCampillo}

Let $R$ be a ring and $A$ be an $R$-algebra. We denote by $\kappa(A)$ its total ring of fractions and by $\overline{A}$ its normalization, i.e. the integral closure of $A$ in $\kappa(A)$. For an element $a\in A$, we will write $\Delta(a) := a\otimes 1- 1\otimes a \in A\otimes_R A$. If $A \inj B$ is an extension of $R$-algebras, we write $I_{\Delta} := \langle \Delta(a) \mid a\in A \rangle \subset B\otimes_R B$. The Lipschitz saturation $A^s_B$ of $A$ in $B$ is defined by 
\[
A^s_B := \Big\{ b\in B \mid \Delta(b)\in \overline{I_{\Delta}} \Big\}
\]
where $\overline{I_{\Delta}}$ is the integral closure of the ideal $I_{\Delta}$. As shown in \cite{Bernard2025RelativeLipSat}, the extension $A\inj A^s_B$ is not finite in general, however the classical use is to consider the Lipschitz saturation of $A$ in $\Bar{A}$. Indeed, usually $A$ is a local ring of germ of holomorphic fonction $\Ocal_{X,0}$ and since locally Lipschitz functions must at least be locally bounded, this naturally leads one to work inside the normalization.  As stated in the introduction practically the only case where one knows how to effectively compute the Lipschitz saturation is the case of complex curves. Indeed, in this case, the Lipschitz saturation is the toric curve whose semigroup is obtained by the algorithmic procedure described in \cite{BrianconGallicoGranger1980DeformationsEqu} or in \cite{FGST2020LipschitzAnAlgebraicApproch}. In order to better understand the Lipschitz saturation of curves in positive characteristic, Campillo reinterpreted this construction as the process of adjoining to $A$ certain specific elements of the normalization $\Bar{A}$. The ring obtained by adjoining all such elements is what he introduced under the name "presaturation" in \cite{Campillo1983OnSatOfCurveSing} and \cite{Campillo1988ArithmeticalAspects}. He proved in \cite[Proposition 1.1]{Campillo1983OnSatOfCurveSing} that this presaturation is always contained in the Lipschitz saturation. Since our goal is to compute Lipschitz saturations and this result provides a method to construct elements belonging to the Lipschitz saturation, we refer to it as "Campillo’s criterion".

We begin by presenting a slightly more detailed version of the proof given in \cite[Proposition 1.1]{Campillo1983OnSatOfCurveSing}. We state it for general extensions $A\inj B$, as this broader setting may be of interest, since the notion of (relative) Lipschitz saturation has also been studied in the context of commutative algebra, for instance in \cite{Lipman1975,DaSilvaNetto2024SurveyLipSat,DaSilvaRibeiro2023, Bernard2025RelativeLipSat}.

\begin{prop}[Campillo's criterion]\label{Prop_CampilloSat}
    Let $A\inj B\inj \kappa(A)$ be an extension of $R$-algebras. Let $p\in A$ and $p_1,\dots, p_r, q_1, \dots q_s \in A^{\times}$ such that 
    $$ \frac{p}{p_i},\text{ }\frac{p}{q_i},\text{ }\frac{p_1\dots p_r}{q_1\dots q_s}\in B $$
    Then $p(p_1\dots p_r)/(q_1\dots q_s) \in A^s_B$.
\end{prop}

\begin{proof}
    Let $p,p_i$ and $q_j$ be as in the assumption of the proposition. We write $w := p(p_1\dots p_r)/(q_1\dots q_s)$. We want to show that $\Delta(w)\in \overline{I_\Delta}$. By the valuative criterion (see \cite{SwansonHuneke2006IntegralClosure} Theorem 6.8.3), it is enough to show that $\Delta(w) \in I_\Delta.V$ for every valuation ring $V$ lying between $B\otimes_R B$ and its total ring of fractions. Let $V$ be such a ring, and let us suppose for now that there exists $a\in V$ such that $(1\otimes q_1\dots q_s) = a(q_1\dots q_s\otimes 1)$.
    
    We consider the following identity
    \begin{equation}
    \begin{split}
        \Delta(w)(q_1\dots q_s\otimes q_1\dots q_s) &= \Delta(p)(p_1\dots p_r\otimes q_1\dots q_s) \\
        &+ \sum_{i=1}^r \Delta(p_i)(p_{i+1}\dots p_r\otimes p p_1\dots p_{i-1} q_1\dots q_s)\\
        &- \sum_{j=1}^s \Delta(q_j)(q_{j+1}\dots q_s\otimes p p_1\dots p_r q_1\dots q_{j-1}) 
    \end{split}
    \end{equation}

    In the ring $V$, we can reorganize the $p_i$ in such a way that there exists $i_0\in \{1,\dots,r\}$ and $(a_i)_{1,\dots,r}\in V^r$ satisfying $1\otimes p_i = a_i(p_i\otimes 1)$ for every $i\leq i_0$ and $p_i\otimes 1 = a_i(1\otimes p_i)$ for every $i>i_0$.
    
    Let $i\leq i_0$, then
    \begin{equation*}
    \begin{split}
        (p_{i+1}\dots p_r\otimes p p_1\dots p_{i-1} q_1\dots q_s) & = (p_{i+1}\dots p_r\otimes p_1\dots p_i q_1\dots q_s)(1\otimes \frac{p}{p_i})\\
        & = (p_1\dots p_r\otimes q_1\dots q_s)(a_1\dots a_i)(1\otimes \frac{p}{p_i})
    \end{split}
    \end{equation*}
    and $(a_1\dots a_i)(1\otimes \frac{p}{p_i}) \in V$ by assumption. Now if $i>i_0$, then 
    \begin{equation*}
    \begin{split}
        (p_{i+1}\dots p_r\otimes p p_1\dots p_{i-1} q_1\dots q_s) &= (1\otimes p_1\dots p_rq_1\dots q_s)(a_{i+1}\dots a_r)(1\otimes \frac{p}{p_i})\\
        &= (q_1\dots q_s\otimes p_1\dots p_r)a(a_{i+1}\dots a_r)(1\otimes \frac{p}{p_i}).
    \end{split}
    \end{equation*}
    By using the same trick on the $q_j$, one can show that $(q_1\dots q_s\otimes p_1\dots p_r)$ divides all the coefficients in front of the $\Delta(q_j)$ in the expression $(1)$. Since, by assumption, $(p_1\dots p_r)/(q_1\dots q_s) \in B$, we get that every coefficient in $(1)$ is divisible by $(q_1\dots q_s\otimes q_1\dots q_s)$ in $V$. Hence, we can write
    $$\Delta(w) = A_0.\Delta(p) + \sum_{i=1}^r B_i\Delta(p_i) - \sum_{j=1}^s C_j\Delta(q_j)$$
    with $A_0, B_i, C_j \in V$, so we get $\Delta(w) \in I_\Delta.V$. However, we had supposed that ${(q_1\dots q_s\otimes 1)\mid (1\otimes q_1\dots q_s)}$ in $V$. If it is not the case, then $(1\otimes q_1\dots q_s) \mid (q_1\dots q_s\otimes 1)$ because $V$ is a valuation ring. Then we use identity (1) on the symmetric expression ${(1\otimes w - w\otimes 1)}{(q_1\dots q_s\otimes q_1\dots q_s)}= -\Delta(w)(q_1\dots q_s\otimes q_1\dots q_s)$. This gives us the same formula as $(1)$  but with symmetric coefficients. Hence, our previous arguments apply symmetrically, and we get that $\Delta(w)\in I_\Delta.V$ for every valuation ring between $B\otimes_R B$ and its total ring of fractions, i.e. $\Delta(w) \in \overline{I_\Delta}$.
\end{proof}

In the remainder of this section, we explain how Campillo's criterion can be used to advantage in the study of semigroups of toric singularities. Let $\Gamma \subset \Z^d$ be an affine semigroup and let $\C\{\Gamma\}$ denotes the ring of convergent series in a neighborhood of $0\in X$ with exponents in $\Gamma$.

\begin{definition}
    Let $(X,0)$ be a germ of a complex analytic space. We say that $(X,0)$ is a germ of a toric singularity if there exists a finitely generated semigroup $\Gamma\subset \Z^d$ contained in a strongly convex cone such that $\Ocal_{X,0} \simeq \C\{\Gamma\}$. Equivalently, let $X_\Gamma$ be the affine toric variety defined by $\Gamma$. Then $(X,0)$ is a toric singularity if it is isomorphic, as germs, to $(X_\Gamma,0)$.
\end{definition}

Let $\sigma$ be a cone such that $\sigma^\vee = \Gamma_\R$. Recall that $\Ocal_{\Bar{X},0} \simeq \C\{\overline{\Gamma}\}$ where $\overline{\Gamma} = \sigma^\vee \cap \Z^d$. We denote by $\Gamma^s$ the sub-semigroup of $\overline{\Gamma}$ of elements whose corresponding monomials belong to $\Ocal_{X^s}$. If $(X,0)$ is a toric singularity with smooth normalization, it is shown in \cite{DuarteGiles2023ToricLipSat} that the Lipschitz saturation $(X^s,0)$ is again toric. In that case $\Ocal_{X^s,0} \simeq \C\{\Gamma^s\}$.


\begin{rmk}\label{Rmk_CampCriterionForSemigroups}
If we apply Campillo's criterion on monomial elements of $\Ocal_{X,0}$, it translates on semigroups in the following way:\\
Let $m, k_1,\dots, k_r, l_1,\dots l_s\in \Gamma$, such that
\[
m-k_i\in \overline{\Gamma},\text{ }m-l_j\in \overline{\Gamma},\text{ and }\sum_i k_i - \sum_j l_j \in \overline{\Gamma}
\]
then $m + \sum_i k_i - \sum_j l_j \in \Gamma^s$. 
\end{rmk}

This gives us a way to produce elements of the Lipschitz saturation of an affine semigroup.
\begin{prop}\label{Prop_mPlusGamma_m}
Let $m\in \Gamma$ and let $\Gamma_m :=\Z\big[(m-\sigma^\vee)\cap \Gamma\big] \cap \sigma^\vee$. Then
\[
 m + \Gamma_m \subset \Gamma^s
\]
\end{prop}

\begin{proof}
Let $n \in \Z[(m-\sigma^\vee)\cap \Gamma]$. Then, there exists some $k_i \in (m-\sigma^\vee)\cap \Gamma$ such that $n = \sum_i k_i - \sum_j l_j$. For each $i$, we have $k_i\in (m-\sigma^\vee)$ and so $m-k_i \in\sigma^\vee$. Hence $m-k_i \in\sigma^\vee \cap \Z^d = \overline{\Gamma}$ and the same is true for every $l_j$. Now, if we suppose $n\in \sigma^\vee$, then, by remark \ref{Rmk_CampCriterionForSemigroups}, we get $m+n \in \Gamma^s$.
\end{proof}

\begin{ex}
    Let us consider the affine semigroup $\Gamma$ in $\Z^2$ generated by the elements $(0,3)$, $(0,4)$, $(0,5)$, $(2,0)$, $(3,0)$, $(1,1)$ and $(1,4)$. As an example, let us consider $m = (3,1)$. We then have $(m-\sigma^\vee)\cap \Gamma = \big\{ (0,0), (1,1), (2,0), (3,0), (3,1) \big\}$ and so $\Z\big[(m-\sigma^\vee)\cap \Gamma\big] = \Z^2$. Hence $\Gamma_m = \Z^2 \cap \R_{\geq 0} = \N^2$ and we get that $(3,2) \in \Gamma^s$ because $(3,2) \in m+\Gamma_m = (3,1)+\N^2$.

    \begin{figure}[h]
    \centering
    \begin{tabular}{ccc}

    \begin{tikzpicture}[scale=0.8]
    
        \foreach \j in {0,...,4}{%
            \foreach \i in {0,...,5}{%
            \node at (\i, \j) {$\bullet$};
            }
        }
    
    
    \node at (0, 1) {{\color{white}$\bullet$}};
    \node at (0, 1) {$\times$};
    \node at (0, 2) {{\color{white}$\bullet$}};
    \node at (0, 2) {$\times$};
    \node at (1, 0) {{\color{white}$\bullet$}};
    \node at (1, 0) {$\times$};
    \node at (1, 2) {{\color{white}$\bullet$}};
    \node at (1, 2) {$\times$};
    \node at (2, 1) {{\color{white}$\bullet$}};
    \node at (2, 1) {$\times$};
    \node at (3, 2) {{\color{white}$\bullet$}};
    \node at (3, 2) {$\times$};
    
    \draw[-] (0,0) -- (5.3,0.01);
    \draw[-] (0,0) -- (0,4.3);

    \draw[color = blue] (3,1) -- (-0.2,1);
    \draw[color = blue] (3,1) -- (3,-0.2);
    \filldraw [fill=blue, fill opacity=0.1, draw = blue] (3,1) rectangle (0,0);
    \node[color = blue] at (3.2,1.2) {$m$};
    
    \end{tikzpicture}

&
    \begin{tikzpicture}[scale=0.5]
        \foreach \j in {-2,...,3}{%
        \node at (1, \j) {};
        }

    \draw[->] (0,0) to [bend left=20] (3,0);
    \end{tikzpicture}
&

    \begin{tikzpicture}[scale=0.8]
    
        \foreach \j in {1,...,4}{%
            \foreach \i in {3,...,5}{%
            \node at (\i, \j) {{\color{blue}$\bullet$}};
            }
        }

        \foreach \j in {0,...,4}{%
            \foreach \i in {0,...,2}{%
            \node at (\i, \j) {$\bullet$};
            }
        }
        
    \node at (3, 0) {$\bullet$};
    \node at (4, 0) {$\bullet$};
    \node at (5, 0) {$\bullet$};

    \node at (0, 1) {{\color{white}$\bullet$}};
    \node at (0, 1) {$\times$};
    \node at (0, 2) {{\color{white}$\bullet$}};
    \node at (0, 2) {$\times$};
    \node at (1, 0) {{\color{white}$\bullet$}};
    \node at (1, 0) {$\times$};
    \node at (1, 2) {{\color{white}$\bullet$}};
    \node at (1, 2) {$\times$};
    \node at (2, 1) {{\color{white}$\bullet$}};
    \node at (2, 1) {$\times$};
    
    \draw[-] (0,0) -- (5.3,0.01);
    \draw[-] (0,0) -- (0,4.3);

    \draw[color = blue] (3,1) -- (5.2,1);
    \draw[color = blue] (3,1) -- (3,4.2);
    \node[color = blue] at (3.3,1.25) {$m$};
    
    \end{tikzpicture}\\

    {\color{blue} $(m-\sigma^\vee)\cap \Gamma$} & & $\Gamma +$ {\color{blue} $m+\Gamma_m$}

\end{tabular}

    \caption{Illustration of Proposition \ref{Prop_mPlusGamma_m}}
    \label{fig:mPlusGamma_m}

\end{figure}
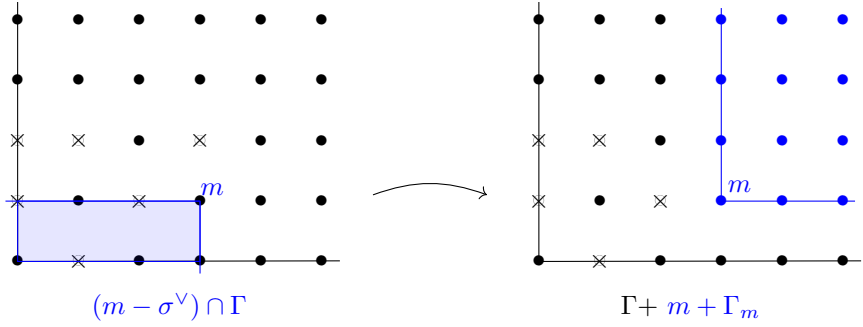
\end{ex}

\begin{rmk}\label{Rmk_finitenessAlgo}
Let $A \inj B$ be a finite extension of rings with $A$ a Noetherian ring. Then one can show that the process of adding elements $b_i \in B$ with $b_{i+1} \notin A[b_1,...,b_i]$ ends after a finite number of steps. This means that the process of adding elements of the form presented in Remark \ref{Rmk_CampCriterionForSemigroups} will end after a finite number of steps. In particular, if $\Gamma$ is an affine semigroup and $X$ is the toric singularity associated to $\Gamma$, then adding elements of $\Gamma^s$ to $\Gamma$ consists in adding monomials elements of $\Ocal_{\Bar{X},0}$ to $\Ocal_{X,0}$. Since $\Ocal_{X,0} \inj \Ocal_{\Bar{X},0}$ is a finite extension, this means that the process of adding elements of the form presented in Remark \ref{Rmk_CampCriterionForSemigroups} will end after a finite number of steps. Since $\Gamma$ is finitely generated, this will produce an affine semigroup and so an affine toric variety.
\end{rmk}

\begin{rmk}
    A simple remark is that if $m\in \Gamma$ and $m'\in \overline{\Gamma} = M\cap \sigma^\vee$ are such that $|m'|\leq |m|$ (for the Euclidean distance), then $m'\notin m+\Gamma_m$. This means that, for any $m\in \Gamma$, the set $(\Gamma + m + \Gamma_m)\setminus \Gamma$ contains no elements smaller than $m$. Therefore, we can add the elements described in Proposition \ref{Prop_mPlusGamma_m} starting from the smallest ones and proceeding according to the order induced by the Euclidean norm. By Remark \ref{Rmk_finitenessAlgo}, this process terminates.
\end{rmk}

\subsection{Newton polyhedrons and support functions}

We now turn to an important tool for describing the semigroup of the Lipschitz saturation: the link between Newton polyhedra and support functions. It is not surprising that Newton polyhedra naturally arise in our study, as they are closely related to the integral closure of monomial ideals.

Let us begin by recalling the correspondence between convex sets and their support functions.

\begin{lemma}[\cite{Rockafellar1997ConvexAnalysis}, Section 13]\label{Lem_ConvexSetSuppFunction}
    Let $C$ be a closed convex set of $\R^d$. Then 
    \[
        C = \Big\{ q\in \R^d \mid \forall v\in \R^d\text{ }\langle q,v \rangle \geq \inf_{p\in C}\{ \langle p ,v\rangle \}\Big\}
    \]
\end{lemma}

\begin{proof}
    It follows, from the correspondence between closed convex sets of $\R^d$ and their support functions, and by the fact that, for all $v\in \R^d$, we have $\inf_{p\in C}\{\langle p,v\rangle \} = -\sup_{p\in C}\{\langle p,-v\rangle \}$.
\end{proof}

We now use this correspondence on a convex cone $\sigma \subseteq \R^d$ to get a characterization of the points which belongs to the Newton polyhedron of a finite set of points, in terms of inequality on scalar products. 
For a finite subset $\{p_i\}_{i\in I}$ of $\R^d$, we denote by $\Ncal_{\sigma}(p_i \mid i\in I) := \Conv(\cup_{i\in I} (p_i + \sigma))$ the Newton polyhedron of $\{p_i\}_{i\in I}$. If $\sigma = \R_{\geq 0}^d$, we simply write $\Ncal(p_i \mid i\in I)$.


\begin{lemma}\label{Lem_NewtonPolCaracScalarProd}
    Let $I$ be a finite subset of $\N$ and let $q,p_i \in \Z^d$, for $i\in I$. Then $q \in \mathcal{N}_{\sigma}(p_i \mid i\in I)$, if and only if
    \[
    \forall v\in \sigma^{\vee}\cap \Z^d \text{ \hspace{0.2cm} }\langle q,v\rangle \geq \min_i\big\{\langle p_i,v\rangle\big\}
    \]
\end{lemma}

\begin{proof}
    Let $\sum_{i\in I} \lambda_i(p_i+z_i)$ be an element of $\Ncal_\sigma(p_i \mid i\in I)$ i.e. $z_i\in \sigma$, $\lambda_i \in \R_{\geq 0}$ and $\sum_i \lambda_i = 1$. Then, for all $v\in \sigma^\vee$, we have \[\langle \sum_{i\in I} \lambda_i(p_i+z_i), v \rangle = \sum_{i\in I} \lambda_i\langle p_i,v \rangle + \sum_{i\in I} \lambda_i\langle z_i,v \rangle \geq \sum_{i\in I} \lambda_i\langle p_i,v \rangle\]
    because every $\lambda_i$ is positive and $z_i \in \sigma,v\in \sigma^\vee$ implies $\langle z_i,v \rangle \geq 0$.
    Moreover, we have 
    \[
    \sum_{i\in I} \lambda_i\langle p_i,v \rangle \geq \left(\sum_{i\in I} \lambda_i \right) \min_{i\in I}\{\langle p_i,v \rangle\} = \min_{i\in I}\{\langle p_i,v \rangle\}
    \]
    So, for all $v\in \sigma^\vee$, we get $\inf_{p\in \Ncal_\sigma(p_i \mid i\in I)}\big\{\langle p_,v \rangle\big\} = \min_{i\in I}\big\{\langle p_i,v \rangle\big\}$ and we get the first implication.
    
    Conversely, let us assume that $\langle q,v\rangle \geq \min_{i\in I}\{\langle p_i,v\rangle\}$, for all $v\in \sigma^\vee\cap\Z^d$. By Lemma \ref{Lem_ConvexSetSuppFunction}, we have to show that $\langle q,v\rangle \geq \inf_{p\in \Ncal_\sigma(p_i \mid i\in I)}\big\{\langle p,v\rangle\big\}$, for all $v\in \R^d$. Suppose there exists $v_0\in \R^d$ such that $\langle q,v_0\rangle < \inf_{p\in\Ncal_\sigma(p_i \mid i\in I)}\{\langle p,v_0\rangle\} = \min_i\{\langle p_i,v_0\rangle\}$. Then there exists $v_0^{\Q}\in\Q^d$ such that $\langle q,v_0^{\Q}\rangle < \min_i\{\langle p_i,v_0^{\Q}\rangle\}$ and so, we can consider a number $N$ verifying $Nv_0^{\Q} \in \Z^d$. We get that $\langle q,Nv_0^{\Q}\rangle = N\langle q,v_0^{\Q}\rangle < N\min_i\big\{\langle p_i,v_0^{\Q}\rangle\big\} = \min_i\big\{\langle p_i,Nv_0^{\Q}\rangle\big\}$. Let us write $v_1 := Nv_0^{\Q} \in \Z^d$.
    
    If $v_1\in \sigma^\vee$, then we get a contradiction. Let us assume that $v_1\notin \sigma^\vee$. Then there exists $x \in \sigma$ such that $\langle x,v_1\rangle < 0$. Since we can consider a small enough neighborhood of $x$ on which the same inequality is verified, we can assume that $x\in \Int(\sigma) \cap \Q^d$, where $\Int$ is the relative interior. After multiplying by a sufficiently large integer, we can assume that $x\in \Int(\sigma \cap \Z^d)$. By \cite[Section 5]{Gubeladze1989AndersonsConj}, we have $x\in \Int(\sigma\cap \Z^d)$ if, and only if, the semigroup $(\sigma\cap \Z^d) + (-x)\N$ equals $\Z^d$. This mean that we can consider $n\in \N$ and $z\in \sigma\cap \Z^d$ such that $q-p_i = z - nx$, for any one of the $p_i$'s. Then $q-p_i + nx \in \sigma\cap \Z^d$ and we get $q+nx \in p_i + \sigma \subseteq \Ncal_\sigma(p_i \mid i\in I)$. So, by Lemma \ref{Lem_ConvexSetSuppFunction}, we get $\langle q+nx,v_1\rangle \geq \inf_{p\in\Ncal_\sigma(p_i\mid i \in I)}\{\langle p,v_1 \rangle\} > \langle q,v_1\rangle$. This implies that $\langle nx,v_1 \rangle > 0$ which is a contradiction.

\end{proof}

\section{A criterion for the elements which are not in the Lipschitz saturation}\label{Section_CriterionForNotinLipSat}

Let $(X,0)$ be a toric singularity and let $(\Bar{X},0)$ be its normalization. Then we have $$\Ocal_{X,0}\simeq \C\{\Gamma\} \subset \Ocal_{\Bar{X},0}\simeq \C\{\Bar{\Gamma}\}$$
with $\Gamma \subset \Z^d$ an affine semigroup and $\Bar{\Gamma} = \Z^d \cap (\R_{\geq 0}\Gamma)$ its saturation. Let $\mathcal{A} = \{\gamma_1,\dots,\gamma_n\} \subset \Z^d$ be such that $\N\mathcal{A} = \Gamma$ and let $\mathcal{B} =\{\delta_1,\dots,\delta_m\} \subset \Z^d$ be such that $\N\mathcal{B} = \Bar{\Gamma}$. We can consider a surjective morphism a semigroups
$$ \begin{array}{cccc}
    \pi: & \N^m & \to & \Bar{\Gamma}\\
     & q & \mapsto & \sum_k q_k \delta_k
\end{array}.$$
and a morphism of $\C$-algebras
$$ \begin{array}{ccc}
    \C[z_1,\dots,z_m] & \to & \C[x_1^{\pm 1},\dots,x_d^{\pm 1}]\\
    z_k & \mapsto & x^{\delta_k}\\
    z^q & \mapsto & x^{\pi(q)}
\end{array}$$
where $z=\prod_k z_k$. If we note $I_{\overline{\Gamma}}$ its kernel, let us recall that
$\Ocal_{\Bar{X},0} \simeq \C\{z_1,\dots,z_m\}/ I_{\Bar{\Gamma}}\C\{z_1,\dots,z_m\}$ (see \cite[Lemma 2.2]{DuarteGiles2023ToricLipSat}). This mean that, since considering an arc $\varphi : (\C,0)\to (\Bar{X},0)$ is equivalent to considering a morphism $\varphi^* : \Ocal_{\Bar{X},0} \to \C\{t\}$, then it is also equivalent to considering a morphism $\C\{z_1,\dots,z_m\} \to \C\{t\}, z_k\mapsto \varphi_k(t)$ where the $\varphi$'s verify the relations of $I_{\Bar{\Gamma}}$.

Recall that the Lipschitz saturation $(X^s,0)$ of $(X,0)$ is the germ of variety such that $\Ocal_{X^s,0} \simeq (\Ocal_{X,0})^s_{\Ocal_{\Bar{X},0}}$ which corresponds to the set of elements $g\in\Ocal_{\Bar{X},0}$ such that $g\otimes_\C 1-1\otimes_\C g \in \Ocal_{\Bar{X}\times\Bar{X},0}$ belongs to the integral closure of the ideal $I_{\Delta} = \langle f\otimes_\C 1-1\otimes_\C f \mid f\in \Ocal_{X,0} \rangle \subset \Ocal_{\Bar{X}\times\Bar{X},0}$. So, by the valuative criterion \cite[Theorem 2.1]{LejeuneTeissierCloture} of integral closures of ideals, an element $g\in\Ocal_{\Bar{X},0}$ is in the Lipschitz saturation if and only if, for every arc $\phi = (\varphi,\varphi') :  (\C,0) \to (\Bar{X}\times \Bar{X},0)$, it satisfies the condition $g\circ\varphi - g\circ\varphi' \in \phi^*(I_{\Delta})$. This condition is of course equivalent to $\ord_t(g\circ\varphi - g\circ\varphi') \geq \ord_t\big(\phi^*(I_{\Delta})\big)$.

Moreover, by definition, the algebra $\C\{\Gamma\}\simeq \Ocal_{X,0}$ is generated by the set $\{ x^{\gamma_i} \}_{i=1,\dots,n}$. This implies that $I_{\Delta} = \langle x^{\gamma_i}\otimes_\C 1 - 1\otimes_\C x^{\gamma_i} \mid \gamma_i\in \mathcal{A} \rangle$  with $x^{\gamma_i} := \prod_k x_k^{\gamma_{i,k}}$. Since $\Gamma \subset \Bar{\Gamma}$, then there exists $p_1,\dots,p_n \in \N^m$ such that $\pi(p_i) = \gamma_i$. So, we can rewrite $I_{\Delta} = \langle x^{\pi(p_i)}\otimes_\C 1 - 1\otimes_\C x^{\pi(p_i)} \mid i=1,\dots,n \rangle$, and so $I_{\Delta} = \langle (\prod_kx^{\delta_k})^{p_i}\otimes 1 - 1\otimes (\prod_kx^{\delta_k})^{p_i} \mid i=1,\dots,n \rangle$.

If $\phi = (\varphi, \varphi') : (\C,0)\to (\Bar{X}\times \Bar{X},0)$ is an arc, we have seen that it is equivalent to consider elements $\varphi_k(t),\varphi'_k(t) \in \C\{t\}$ verifying the relations of $I_{\Bar{\Gamma}}$. Each $\varphi_{k}$ is the image of the generator $x^{\delta_k}\in \C\{\Bar{\Gamma}\}$. So we get
$\phi^*(I_\Delta) = \langle \varphi^{p_i} - (\varphi')^{p_i} \mid p_i\in \N^m \text{ such that }\pi(p_i)=\gamma_i \rangle$ where, for $p\in \N^m$, we write $\varphi^p := \prod_{k=1}^m \varphi_k^{p_k}$. \\

\begin{rmk}\label{Rmk_ValuativeCriterionFroMonomials}
To conclude, every monomial in $\C\{\Bar{\Gamma}\}$ can be written $x^{\pi(q)}$ for a non-unique element $q\in \N^m$, and we have that $x^{\pi(q)} \in \C\{\Gamma\}^s$ if and only if, for every arc $\phi = (\varphi, \varphi') : (\C,0)\to (\Bar{X}\times \Bar{X},0)$, we have 
$$ \ord_t(\varphi^q - (\varphi')^q) \geq \min_{i=1\dots,n} \big\{ \ord_t(\varphi^{p_i}-(\varphi')^{p_i}) \big\} $$
where the $p_i\in \N^m$ are such that $\pi(p_i)=\gamma_i$.\\
\end{rmk}


\begin{notation} We will denote by $\Gamma^s$ the   representing the monomials of $\C\{\Bar{\Gamma}\}$ which belong to $\Ocal_{X^s,0}$. By \cite[Theorem 3.2]{DuarteGiles2023ToricLipSat} we know that $\Ocal_{X^s,0} \simeq \C\{\Gamma^s\}$, but only in the case where $\Bar{X}$ is smooth, which corresponds to the situation where $\C\{\Bar{\Gamma}\} \simeq \C\{\N^d\}$ or, in other words, $\Ocal_{\Bar{X},0}\simeq \C\{z_1,\dots,z_m\}$.\\

Let $\phi = (\varphi,\varphi') : (\C,0) \to (X\times X,0)$ be an arc, we will always imply that $\varphi$ and $\varphi' : (\C,0) \to (X,0)$ are of the form $\varphi_k(t) = a_kt^{v_k} + \cdots$ and $\varphi'_k(t) = (a_k)'t^{v_k'} + \cdots$. As said previously, for $p\in\N^m$, we write $\varphi^p := \prod_{k=1}^m \varphi_k^{p_k}$, idem for $\varphi',a$ and $a'$. Moreover, we will call $v$ and $v'$ the vectors $(v_k)_k $ and $(v'_k)_k$ and, for all $p$, we will write $\langle p,\tilde{v}\rangle := \min \big\{ \langle p,v\rangle , \langle p,v'\rangle \big\}$.
\end{notation}

\begin{rmk}\label{Rmk_JumpOfValuationIFF}
    Let $p\in \N^m$. Then 
    \[
    \begin{array}{ccc}
         \ord\big(\varphi^p - (\varphi')^p\big) > \langle p,\tilde{v}\rangle & \iff &  \langle p, v\rangle = \langle p, v'\rangle \text{ and }a^p - (a')^p = 0 \medskip\\
         & \iff & p\in (v-v')^\perp\text{ and }a^p - (a')^p = 0
    \end{array}   
    \]
\end{rmk}

According to remark \ref{Rmk_ValuativeCriterionFroMonomials}, showing that an element does not belong to the Lipschitz saturation consists in exhibiting a carefully chosen arc. In order to do so, we need to control the cancellation of the leading coefficients. This will be possible thanks to the following lemma. This is the reason why different lattices naturally arise in the description of the semigroup of the Lipschitz saturation.

\begin{lemma}\label{Lem_SolutionSystem}
    Let $p_1,\dots,p_n,q \in \Z^d$. Then $q\in \sum_i p_i\Z$ if and only if, for all $a$ and $a'\in (\C^*)^{d}$, we have 
    \[ \left\{\begin{array}{l}
          a^{p_1}-(a')^{p_1} = 0 \\
          \vdots\\
          a^{p_n}-(a')^{p_n} = 0
    \end{array}\right. \implies a^{q}-(a')^{q} = 0
    \]
\end{lemma}

\begin{proof}

    First, one can see that
    \[
    \exists (a,a')\in (\C^*)^{2d}\text{, } a^p-(a')^p=0 \iff \exists a\in \C^{d}\text{, } a^p-1=0.
    \]
    
    Let $q\in \sum_i p_i\Z$. Then, there exists some $k_i\in \Z$ such that $q = \sum_i k_ip_i$. Let  $a\in\C^d$ such that $a^{p_i}=1$, for all $i$. Then 
    \[
    a^q = \prod_{k=1}^d (a_k)^{\sum_i k_ip_{i,k}} = \prod_{k=1}^d \prod_{i=1}^n (a_k^{p_{i,k}})^{k_i} = \prod_{i=1}^n \big(\prod_{k=1}^d a_k^{p_{i,k}}\big)^{k_i} = \prod_{i=1}^n (a^{p_i})^{k_i} = 1
    \]

    Conversely, let us denote by $I$ the ideal of $\C[x_i^{\pm 1},y_i^{\pm 1}]_{i=1,\dots,d}$ generated by the elements $x^{p_i}-1$ and, following the notation of \cite{EisenbudSturmfels1996BinomialIdeals} Section 2, let us consider the Laurent binomial ideal $I(\rho) = \langle x^m - \rho(m) \mid m\in L_{\rho}\rangle$ where $L_{\rho}$ is the lattice generated by the elements $p_i$ in $\Z^d$ and $\rho : \Z^d \to \C^*$ is the constant morphism identically equal $1$. Let us assume that we have the inclusion $\Zcal(I) \subseteq \Zcal(x^q-1)$ of Zariski sets. Then, by the Nullstellensatz, we have $\sqrt{\langle x^q-1 \rangle} \subset \sqrt{I}$. Moreover, we have $I \subseteq I(\rho)$ by definition and, since we work over an algebraically closed field of characteristic $0$, then, by \cite{EisenbudSturmfels1996BinomialIdeals} Corollary 2.2, the ideal $I(\rho)$ is radical. Hence $\sqrt{I} \subseteq I(\rho)$. This implies that $x^q-1 \in I(\rho)$ and so, by \cite{EisenbudSturmfels1996BinomialIdeals} Theorem 2.1 (a), we get $q\in L_{\rho} = \sum_i p_i\Z$.
\end{proof}

By combining the characterization of Newton polyhedra in terms of scalar products with Lemma \ref{Lem_SolutionSystem}, we can construct suitably chosen arcs to obtain a criterion insuring that a given monomial element does not belong to the Lipschitz saturation.

\begin{prop}\label{Prop_CriterionToNotBelongInGammas}
    Let $(X,0)$ be a toric singularity, let $\Gamma\subset \Z^d$ by its associated semigroup, and let $\sigma$ be the cone generated by $\Gamma$. Let $p_1,\dots,p_n \in \N^m$ be such that $\pi(p_i)$ are generators of $\Gamma$ and let $q\in \N^m$. If there exists $I\subseteq \{1,\dots,n\}$ such that $\pi(q)\notin \Ncal_\sigma(\pi(p_i) \mid i\in I)$ and $q\notin \displaystyle{\sum_{i\notin I} p_i\Z}$, then $\pi(q)\notin \Gamma^s$.
\end{prop}

\begin{proof}
 
    By Lemma \ref{Lem_SolutionSystem}, we can consider $(a,a') \in (\C^*)^{2d}$ such that $a^{q}-(a')^{q} \neq 0$ and $a^{p_i}-(a')^{p_i} = 0$, for every $i\notin I$. Moreover, by lemma \ref{Lem_NewtonPolCaracScalarProd}, since $\pi(q)\notin \Ncal(\pi(p_i) \mid i\in I)$, we can consider $v\in \sigma^\vee \cap \Z^d$ such that 
    \[
    \langle \pi(q),v \rangle < \min_{i\in I}\big\{\langle \pi(p_i), v\rangle\big\}
    \]
    Let $\phi = (\varphi,\varphi') : (\C,0) \to (\Bar{X}\times \Bar{X},0)$ be the arc defined by $\varphi_k(t) =  a_k t^{\langle v, \delta_k\rangle}$ and $\varphi'_k(t) = a'_k t^{\langle v, \delta_k\rangle}$. This arc is well-defined because $v\in \sigma^\vee \cap \Z^d$ implies that $\langle v, \delta_k \rangle \in \N$ and since, for every $k$, we can write $a_k t^{\langle v, \delta_k\rangle} = \big(\prod_{j\neq k} t^{v_j}\times a_k^{1/\delta_{k,k}}t^{v_k}\big)^{\delta_k}$, it implies that the $\varphi_k$'s satisfy the same relations as the $x^{\delta_k}$ defined at the beginning of Section \ref{Section_CriterionForNotinLipSat}. Then, for any $p\in \N^m$, we have $$\varphi^p = \prod_{k=1}^m \varphi_k^{p_k} = \prod_{k=1}^m a_k^{p_k}t^{p_k\langle v,\delta_k \rangle} = (\prod_{k=1}^m a_k^{p_k})t^{\sum_k p_k \langle v, \delta_k \rangle} = a^{p}t^{\langle v,\pi(p)\rangle}.$$
    So $\varphi^p - (\varphi')^p = (a^{p} - (a')^{p})t^{\langle v, \pi(p) \rangle}$. We get that $\ord(\varphi^{p_i}-\varphi'^{p_i}) = \infty$, for all $i\notin I$, and $\ord(\varphi^{q}-\varphi'^{q}) = \langle \pi(q),v \rangle < \min_{i\in I}\big\{\langle \pi(p_i), v\rangle\big\} \leq \min_{i\in I}\big\{\ord(\varphi^{p_i}-\varphi'^{p_i})\big\}$. Hence $\varphi^{q}-\varphi'^{q} \notin \phi^*(I_\Delta)$ and $q\notin \Gamma^s$.


\end{proof}

\section{Description of the Lipschitz saturation of a toric singularity}\label{SecMain}

In this section, we show how to explicitly compute the semigroup of the Lipschitz saturation of a toric singularity with smooth normalization. This result can be viewed as a generalization of the one-dimensional case studied in \cite[Theorem VI.1.6]{BrianconGallicoGranger1980DeformationsEqu} and the proof is inspired by the strategy used there, though it introduces additional arguments to handle the complexities arising in higher dimensions.

As in the one-dimensional case, we would like to get a better understanding of $\ord(\phi^*(I_\Delta))$. We start by looking at the elements for which the order does not jump.

\begin{definition}
    Let $(X,0)$ be a toric singularity, let $\phi = (\varphi,\varphi') : (\C,0) \to (\overline{X}\times \overline{X},0)$ be an arc and let $p_1,\dots,p_n \in \N^m$ be such that $x^{\pi(p_i)}$ are generators of $\Gamma$. Then we define
    \[
    M_\phi := \Big\{ p_i \mid \ord\big(\varphi^{p_i} - (\varphi')^{p_i}\big) = \langle p_i, \tilde{v}\rangle\Big\}
    \]
\end{definition}

The importance of considering $M_\phi$ lies in the following proposition.

\begin{prop}\label{Prop_ElementInMphiAreOk}
    Let $(X,0)$ be a toric singularity, let $\phi = (\varphi,\varphi') : (\C,0) \to (\overline{X}\times \overline{X},0)$ be an arc and let $q\in\N^m$. Then
    \[
    q\in \Ncal(M_\phi) \implies \ord\big(\varphi^q - (\varphi')^q\big) \geq \ord\big(\phi^*(I_\Delta)\big)
    \]
\end{prop}

\begin{proof}
    Let $q\in \Ncal(M_\phi)$. Then, by Lemma \ref{Lem_NewtonPolCaracScalarProd}, we have 
    \[
    \langle q,v\rangle \geq \min_{p\in M_\phi}\big\{\langle p,v\rangle\big\} \geq \min_{p\in M_\phi}\big\{\langle p,\tilde{v}\rangle\big\}
    \]
    and the same is true for $v'$. Hence $\langle q,\tilde{v}\rangle \geq \min_{p\in M_\phi}\big\{\langle p,\tilde{v}\rangle\big\}$. Then 
    \[
    \begin{array}{llll}
        \ord(\varphi^q - (\varphi')^q) & \geq & \langle q,\tilde{v}\rangle & \geq \min_{p\in M_\phi}\big\{\langle p,\tilde{v}\rangle\big\} = \min_{p\in M_\phi}\big\{\ord(\varphi^p - (\varphi')^p)\big\}\medskip\\
        &&& \geq \min_{p\in \pi^{-1}(\Gamma)}\big\{\ord(\varphi^p - (\varphi')^p)\big\} = \ord(\phi^*(I_\Delta))
    \end{array}
    \]
\end{proof}

Considering the previous proposition, we only need to understand, for a given arc $\phi$, what happens to the order of $\ord(\varphi^q-(\varphi')^q)$ when $q$ does not belong to $\Ncal(M_\phi)$. In order to do that, we first need to show the following technical lemma.

\begin{lemma}\label{Lem_OnePlusPhiInBinomial}
    Let $\{\varphi_i\}_i$ and $\{\varphi'_i\}_i \in \C\{t\}^m$ and let us write $\varphi_i-\varphi_i'=b_it^{\alpha_i}(1+G_i)$ with $G_i \in \C\{t\}$. Then, there exists $F_1,\dots, F_d \in \C\{t\}$ such that, for all $q\in \N^m$, we have
    \[
    (1+\varphi)^q - (1+\varphi')^q = \sum_{i=1}^m q_ib_it^{\alpha_i}(1+F_i)
    \]
\end{lemma}

\begin{proof}
    Let $q\in \N^m$. For every $i\in\{1,\dots,m\}$, we have $(1+\varphi_i)^{q_i} = \sum_{k=0}^{q_i}\binom{q_i}{k} \varphi_i^k$. Hence $(1+\varphi)^{q} = \prod_{i=1}^m (1+\varphi_i)^{q_i} = \sum_{k_1,\dots,k_d} (\prod_{i=1}^m\binom{q_i}{k})(\prod_{i=1}^d\varphi_i^{k_i})$. Then we get 
    \[
    \begin{array}{lll}
        (1+\varphi)^q - (1+\varphi')^q & = & \sum_{k_1,\dots,k_d} (\prod_{i=1}^m\binom{q_i}{k_i})\Big(\prod_{i=1}^m\varphi_i^{k_i}-\prod_{i=1}^m(\varphi')_i^{k_i}\Big). \\
    \end{array}
    \]
    By using the identity $AB-CD = (A-C)B+C(B-D)$, for every $\{k_i\}_i$, we can consider some $A_i\in \C\{t\}$ such that $\prod_{i=1}^m\varphi_i^{k_i}-\prod_{i=1}^m(\varphi'_i)^{k_i} = \sum_{i=1}^m A_i(\varphi_i^{k_i}-(\varphi'_i)^{k_i})$. Moreover, since $\varphi_i-\varphi'_i$ divides $\varphi_i^{k_i}-(\varphi'_i)^{k_i}$, then we can consider some $B_i\in \C\{t\}$ such that $\sum_{i=1}^m A_i(\varphi_i^{k_i}-(\varphi'_i)^{k_i}) = \sum_{i=1}^m B_i(\varphi_i-\varphi'_i)$. Then we have 
    $$
    \begin{array}{ll}
        & \sum_{k_1,\dots,k_m} (\prod_{i=1}^m\binom{q_i}{k_i})\Big(\prod_{i=1}^m\varphi_i^{k_i}-\prod_{i=1}^m(\varphi'_i)^{k_i}\Big) \medskip\\
        =& \sum_{i=1}^m q_i (\varphi_i-\varphi'_i) + \sum_{\{k_j\}_j \neq e_i} (\prod_{i=1}^m\binom{q_i}{k_i})\Big(\prod_{i=1}^m\varphi_i^{k_i}-\prod_{i=1}^m(\varphi'_i)^{k_i}\Big) \medskip\\
        =& \sum_{i=1}^m q_i (\varphi_i-\varphi'_i) + \sum_{\{k_j\}_j \neq e_i} (\prod_{i=1}^m\binom{q_i}{k_i})\sum_{i=1}^m B_i(\varphi_i-\varphi'_i) \medskip\\
        =& \sum_{i=1}^m q_i (\varphi_i-\varphi'_i) + \sum_{i=1}^m (\varphi_i-\varphi'_i)C_i \text{ for some }C_i\in \C\{t\} \medskip\\
        =& \sum_{i=1}^m q_i (\varphi_i-\varphi'_i)(1+D_i) \text{ for some }D_i\in \C\{t\}\\
    \end{array}
    $$
    where $e_i$ is the canonical basis of $\R^m$. Now, we can write $\varphi_i-\varphi'_i = b_it^{\alpha_i}(1+G_i)$ for some $b_i\in \C$ and some $G_i \in \C\{t\}$. Note that the $b_i$'s can be $0$ if $\varphi_i=\varphi'_i$. Otherwise $\alpha_i = \ord(\varphi_i-\varphi'_i)$. Finally, we have obtained that
    $$(1+\varphi)^q - (1+\varphi')^q = \sum_{i=1}^m q_i b_it^{\alpha_i}(1+F_i)$$
    with $F\in \C\{t\}^m$ such that $1+F_i = (1+G_i)(1+D_i)$.
\end{proof}

If an element $q\in\N^m$ does not belong to $\Ncal(M_{\phi})$, then the order of $\varphi^q-(\varphi')^q$ jumps, that is, it is higher than $\langle q,\Tilde{v}\rangle$. Hence, we would like to have more information on the order when this jump occurs. This is exactly the content of the next proposition, which describes all the possible jumps of the order of $\varphi^q-(\varphi')^q$. More precisely, it states that these jumps occur along strata and that comparing the orders of two elements on the same stratum is the same as comparing the scalar products $\langle q,\Tilde{v}\rangle$.

\begin{prop}\label{Prop_ChainOfJumps}
    Let $(X,0)$ be a toric singularity and let $\Gamma\subseteq \Z^d$ be its associated semigroup. Let $\phi = (\varphi,\varphi') : (\C,0) \to (\Bar{X}\times \Bar{X},0)$ be an arc. Then there exists a chain of vector spaces $\R^m = V_0 \supsetneq V_1 \supsetneq \dots \supsetneq V_l \supsetneq V_{l+1} = \emptyset$ and there exists real numbers $\delta_0 < \delta_1 < \dots < \delta_{l} = +\infty$ such that, for all $q\in \N^m$ verifying $\ord\big(\varphi^q - (\varphi')^q\big) > \langle q,\tilde{v}\rangle$ and for all $k\in \{0,\dots,l\}$, we have
    $$
    \begin{array}{llll}
        && (1) & \ord\big(\varphi^q - (\varphi')^q\big) \geq \langle q,\tilde{v}\rangle + \delta_k \iff q\in V_k  \\
        &&&\\
        \text{and} && (2) & \ord\big(\varphi^q - (\varphi')^q\big) = \langle q,\tilde{v}\rangle + \delta_k \iff q\in V_k\setminus V_{k+1}
    \end{array}
    $$
\end{prop}

\begin{rmk}
    Since $V_l$ can be $0$, we always have $\delta_l = +\infty$.
\end{rmk}

\begin{proof}
    First, since we have assumed that $\ord\big(\varphi^q - (\varphi')^q\big) > \langle q,\tilde{v}\rangle$, then we can write
    \[
    \begin{array}{lll}
        \varphi^q - (\varphi')^q & = & \prod_{i=1}^m a_i^{q_i}t^{q_iv_i}(1+\psi_i)^{q_i} - \prod_{i=1}^m (a_i')^{q_i}t^{q_iv'_i}(1+\psi'_i)^{q_i} \medskip\\
         & = & a^q t^{\langle q,v \rangle} \Big[\prod_{i=1}^m (1+\psi_i)^{q_i} - \prod_{i=1}^m (1+\psi'_i)^{q_i} \Big] \medskip\\
         & = & a^q t^{\langle q,v \rangle} \Big[ (1+\psi)^{q} - (1+\psi')^{q} \Big].
    \end{array}
    \]
    By Lemma \ref{Lem_OnePlusPhiInBinomial}, we can write
    \[
    \varphi^q - (\varphi')^q = a^q t^{\langle q,v \rangle}\sum_{i=1}^m q_i b_it^{\alpha_i}(1+F_i).
    \]
    where $\psi_i-\psi'_i = b_it^{\alpha_i}(1+G_i)$ for some $G_i\in \C\{t\}$. Let $\delta_0 := \min_i\{\alpha_i\}$, then we have that $\ord\big(\varphi^q - (\varphi')^q\big) \geq \langle q,v \rangle + \delta_0$ and that this inequality is strict if and only if $\sum_{i\in I_1}q_ib_i = 0$ where $I_1 = \big\{i \mid \alpha_i = \delta_0 \big\}$. If $\delta_0 = +\infty$, then $l=0$ and the theorem holds. Otherwise we have that the element $c\in \C^m$ defined by $c_i = b_i$ if $i\in I_1$ and $c_i = 0$ if $i\notin I_1$, is non zero. The condition for having a strict inequality is equivalent to asking for $q$ to be an element of the kernel of the linear map $\langle .,c \rangle : x \mapsto \langle x,c \rangle$ which goes from $\R^m$ to $\C$. Let $E_1$ be $\ker(\langle .,c \rangle)$, which is a proper vector subspace of $\R^m$ since $c\neq 0$. Then, by setting $V_0=\R^m$ and $V_1 = E_1$, we have
    $$ \ord\big(\varphi^q - (\varphi')^q\big) = \langle q,v \rangle + \delta_0 \iff q\in V_0\setminus V_1$$
    Let us see now what happens to the order on $E_1$. Since we have $E_1^\perp \neq 0$, then we can consider a non zero element $h \in E_1^\perp$ such that, for every $q\in E_1$, we have $\sum_{i=1}^m q_i h_i = 0$. Since $h\neq 0$, we can consider $i_0$ such that $q_{i_0} = \sum_{i=1}^m (-h_i/h_{i_0})q_i$. For every $i\neq i_0$, let $\Tilde{\psi_i} \in \C\{t\}$ be such that 
    \[
    1+\Tilde{\psi_i} = (1+\psi_i)(1+\psi_{i_0})^{-h_i/h_{i_0}}
    \]
    and write $\Tilde{q} = (q_1,\dots,q_{i_0-1},q_{i_0+1},\dots,q_m)$. Then we have $(1+\psi)^{q} - (1+\psi')^{q} = (1+\Tilde{\psi})^{\Tilde{q}} - (1+\Tilde{\psi'})^{\Tilde{q}}$. So, by Lemma \ref{Lem_OnePlusPhiInBinomial}, we can write  
    \[
    \varphi^q - (\varphi')^q = a^q t^{\langle q,v \rangle}\sum_{i=1, i\neq i_0}^m q_i \Tilde{b_i}t^{\beta_i}(1+\Tilde{F_i}).
    \]
    where $\Tilde{\psi_i}-\Tilde{\psi'_i} = \Tilde{b_i}t^{\beta_i}(1+\Tilde{G_i})$ for some $\Tilde{G_i}\in \C\{t\}$.
    Let $\delta_1 := \min_i\{\beta_i\}$, then we have that $\ord\big(\varphi^q - (\varphi')^q\big) \geq \langle q,v \rangle + \delta_1$ and that this inequality is strict if and only if $\sum_{i\in I_2}q_i\Tilde{b_i} = 0$ where $I_2 = \big\{i \mid \beta_i = \delta_1 \big\}$. If $\delta_1 = +\infty$, then the theorem holds with $l=1$. Otherwise, the element $\Tilde{c}\in \C^{m-1}$ defined by $\Tilde{c_i} = \Tilde{b_i}$ if $i\in I_2$ and $\Tilde{c_i} = 0$ if $i\notin I_2$, is non zero. The condition for having a strict inequality is equivalent to asking for $\Tilde{q}$ to be an element of the kernel of the linear map $\langle .,\Tilde{c} \rangle : x \mapsto \langle x,\Tilde{c} \rangle$ which goes from $\R^{m-1}$ to $\C$. Let $E_2$ be $\ker(\langle .,\Tilde{c} \rangle)$, which is a proper vector subspace of $\R^{m-1}$ since $\Tilde{c}\neq 0$. Then, by setting $V_0=\R^m$, $V_1 = E_1$ and $V_2 = E_1\cap(E_2\oplus \R e_{i_0})$ we have
    \[
    \ord\big(\varphi^q - (\varphi')^q\big) \geq \langle q,v\rangle + \delta_k \iff q\in V_k 
    \]
    and
    \[
    \ord\big(\varphi^q - (\varphi')^q\big) = \langle q,v\rangle + \delta_k \iff q\in V_k\setminus V_{k+1}
    \]
    for $k=0,1$. Now, we can continue this process, by taking a non zero element of $E_2^\perp$ in $\R^{m-1}$. We know that at some point it will stop because the $E_i$'s are of smaller and smaller dimension. Then we will obtain a chain 
    $\R^m = V_0 \supseteq V_1 \supseteq \dots \supseteq V_l$ and some positive numbers $\delta_i$'s with $\delta_l = +\infty$ verifying proprieties $(1)$ and $(2)$. Note that if, for some $k$, we have $V_k = V_{k+1}$, then property $(2)$ is empty, and we can just remove $V_k$ from the chain. By doing so, we obtain strict inclusions between the $V_k$'s.  Moreover, those properties imply, for all $k$, that $\delta_{k} < \delta_{k+1}$.
\end{proof}

\begin{rmk}
By applying Propositions \ref{Prop_ElementInMphiAreOk} and \ref{Prop_ChainOfJumps} to elements $p_1,\dots,p_n\in \N^m$ such that $\{\pi(p_i)\}_i$ are generators of $\Gamma$, we can describe, for a given arc $\phi$, the order of the pullback of $I_{\Delta}$ as follows:
\[
\ord\big(\phi^*(I_{\Delta})\big) = \min_k\Big\{ \min_{p_i\in\Ncal(M_\phi)}\{\langle p_i,\Tilde{v}\rangle\}; \min_{p_i\notin\Ncal(M_\phi), p_i\in V_k\setminus V_{k+1}}\{\langle p_i,\Tilde{v}\rangle + \delta_k\}\Big\}
\]
\end{rmk}

We now prove the main theorem of this paper, which gives the description of the semigroup of the Lipschitz saturation of a toric singularity.

\begin{theorem}\label{Th_Main}
    Let $(X,0)$ be a toric singularity, let $\Gamma\subseteq \Z^d$ be its associated semigroup, and let $\sigma$ be the cone generated by $\Gamma$. Let $p_1,\dots,p_n \in \N^m$ be such that $\{\pi(p_i)\}_{i}$ are generators of $\Gamma$ and let $q\in \N^m$. Then the following statements are equivalent:
    \begin{enumerate}
        \item The element $\pi(q)$ belongs to $\Gamma^s$.
        \item For all $I \subseteq \{1,\dots,n\}$ such that $\pi(q)\notin \Ncal_{\sigma}(\pi(p_i) \mid i\in I)$, then $q\in \displaystyle{\sum_{i\notin I} p_i\Z} \cup \{0\}$.
    \end{enumerate}
\end{theorem}

\begin{proof}
    The implication $1\implies 2$ is the contraposition of Proposition \ref{Prop_CriterionToNotBelongInGammas}, so we just have to prove $2\implies 1$. Let $q\in \N^m$ be an element verifying $2.$, and let $(\varphi, \varphi') : (\C,0) \to (\overline{X}\times \overline{X},0)$ be an arc. We want to prove 
    \[
    \ord(\varphi^q - (\varphi')^q) \geq \min_{i=1,\dots,n}\big\{\ord(\varphi^{p_i} - (\varphi')^{p_i})\big\}
    \]
    First, remark that, for every set $I$, the condition $\pi(q)\notin \Ncal_{\sigma}(\pi(p_i) \mid i\in I)$ implies $q\notin \Ncal(p_i \mid i\in I)$. By Proposition \ref{Prop_ElementInMphiAreOk}, we can assume $q\notin \Ncal(M_{\phi})$. Let us consider $I\subseteq \{1,\dots,n\}$ such that $i\in I$ if and only if $p_i\in \Ncal(M_\phi)$ and let $J=\{1,\dots,n\}\setminus I$. Then $q\notin \Ncal(p_i \mid i\in I)$ and, by $2.$, we get $q\in \sum_{j\in J} p_j\Z \cup \{0\}$. Since, for all $j\in J$, we have $p_j \notin M_\phi$, then $\{p_j\}_j\subset (v-v')^\perp$, by remark \ref{Rmk_JumpOfValuationIFF}. So $q\in \sum_{j\in J} p_j\Z \cup \{0\}$ implies that $q\in (v-v')^\perp$. By Lemma \ref{Lem_SolutionSystem}, it also implies that $a^q-(a')^q=0$ and so we have $\ord(\varphi^q - (\varphi')^q) > \langle q,\tilde{v}\rangle$. By applying Proposition \ref{Prop_ChainOfJumps}, we can consider a chain $\R^m = V_0 \supsetneq \dots \supsetneq V_{l} \supsetneq V_{l+1}=\emptyset $ of vector spaces and a sequence of numbers $ \delta_0 < \dots < \delta_{l} = +\infty$ such that, for all $k\in \{0,\dots,l\}$, we have
    \[
    \ord\big(\varphi^q - (\varphi')^q\big) \geq \langle q,\tilde{v}\rangle + \delta_k \iff q\in V_k 
    \]
    and
    \[
    \ord\big(\varphi^q - (\varphi')^q\big) = \langle q,\tilde{v}\rangle + \delta_k \iff q\in V_k\setminus V_{k+1}.
    \]
    
    Let us consider $k_0$ such that $q\in V_{k_0}\setminus V_{k_0+1}$ and let $\Lambda_{k_0} \subseteq \{1,\dots,n\}$ be such that $i\in \Lambda_{k_0}$ if and only if $p_i \notin V_{k_0+1}$. If $q \in \Ncal(p_i \mid i\in \Lambda_{k_0})$, then, by Lemma \ref{Lem_NewtonPolCaracScalarProd}, we get 
    $$ 
    \ord(\varphi^q - (\varphi')^q) \geq \langle q,v \rangle + \delta_{k_0} \geq \min_{i\in \Lambda_{k_0}}\big\{\langle p_i,v \rangle + \delta_{k_0} \big\} \geq \min_{i\in \Lambda_{k_0}} \big\{ \ord(\varphi^{p_i} - (\varphi')^{p_i}) \big\}
    $$
    and so $\pi(q)\in \Gamma^s$. If $q \notin \Ncal(p_j \mid j\in \Lambda_{k_0})$, then, by $2.$, we get $q\in \sum_{j\notin \Lambda_{k_0}} p_j\Z \cup \{0\}$. In the case where $\{ p_j \mid j\notin \Lambda_{k_0}\}$ is empty, then $q = 0$, otherwise we get a contradiction to the assumption $q\notin V_{k_0+1}$, since $\{ p_j \mid j\notin \Lambda_{k_0}\} \subset V_{k_0+1}$ by definition of $\Lambda_{k_0}$.
\end{proof}

\begin{ex}\label{Ex_KeyEx}
    Let us consider the semigroup generated by $p_1,\dots,p_7 \in \N^2$ as in Figure \ref{fig:KeyExample}. Note that its saturation is $\N^2$, so $\pi(p_i)=p_i$. First, by Campillo's criterion, we have $\sum_{i=1,2,3,7} p_i\Z = (2,0)\Z\oplus(0,2)\Z$ on $\Ncal(p_3)$ and $\sum_{i=1,2,4} p_i\Z = (1,0)\Z\oplus(0,4)\Z$ on $\Ncal(p_4)$.
    Now, let us consider $q:=(10,8)$. This element is not in $\Ncal(p_3,p_7,p_5,p_6)$, but it belongs to $\sum_{i=1,2,4} p_i\Z$. It is also not in $\Ncal(p_4,p_5,p_6) = \Ncal(p_4)$, but it belongs to $\sum_{i=1,2,3,7} p_i\Z$. Then, by Theorem \ref{Th_Main}, we get $q\in \Gamma^s$. This was in fact the only missing element of $\sum_{i=1,2,4}p_i\Z \cap \sum_{i=1,2,3,7} p_i\Z$ on $\Ncal(p_3,p_4) \setminus (\Ncal(p_3)\cup\Ncal(p_4))$, so we have computed the Lipschitz saturation of the semigroup $\langle p_1,\dots, p_7\rangle$.\\
    
    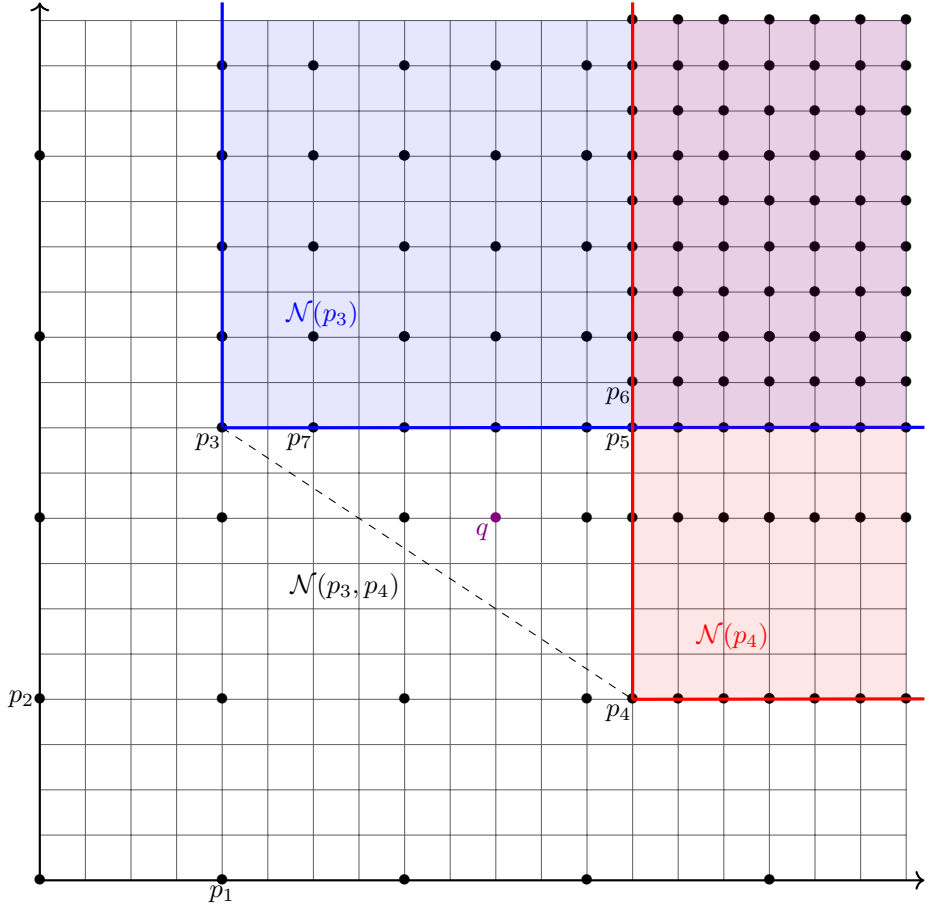
\begin{figure}[h]
    \centering
    \begin{tikzpicture}[scale=0.6]

    \def \xmax {19};
    \def \ymax {19};
    
    \draw[opacity = 0.6] (0,0) grid (\xmax,\ymax);

    \def \pun {4};
    \def \pdeux {4};

    \node[] at (\pun,-0.3) {$p_1$};
    \node[] at (-0.4,\pdeux) {$p_2$};
    
        \foreach \j in {0,...,4}{%
            \foreach \i in {0,...,4}{%
            \node at (\pun*\i,\pdeux*\j) {$\bullet$};
            }
        }
    
    \def \xtrois {4};
    \def \ytrois {10};
    \node at (\xtrois-0.3,\ytrois-0.3) {$p_3$};
    \foreach \j in {0,...,4}{%
            \foreach \i in {0,...,4}{%
            \node at (\xtrois+2*\i,\ytrois+2*\j) {$\bullet$};
            }
        }

    \def \xsept {6};
    \def \ysept {10};
    \node at (\xsept-0.3,\ysept-0.3) {$p_7$};
    
    \def \xquatre {13};
    \def \yquatre {4};
    \node at (\xquatre-0.3,\yquatre-0.3) {$p_4$};
    \foreach \j in {0,...,2}{%
            \foreach \i in {0,...,6}{%
            \node at (\xquatre+\i,\yquatre+4*\j) {$\bullet$};
            }
        }
    
    \def \xcinq {\xquatre};
    \def \ycinq {\ytrois};
    \def \xsix {\xquatre};
    \def \ysix {\ytrois+1};
    \node at (\xcinq-0.3,\ycinq-0.3) {$p_5$};
    \node at (\xsix-0.3,\ysix-0.3) {$p_6$};
    \foreach \j in {0,...,9}{%
            \foreach \i in {0,...,6}{%
            \node at (\xcinq+\i,\ycinq+\j) {$\bullet$};
            }
        }
    
    \node[color = blue] at (\xtrois+2.2,\ytrois+2.5) {$\Ncal(p_3)$};
    \draw[color = blue, very thick] (\xtrois,\ytrois) -- (\xmax+0.4,\ytrois+0.01);
    \draw[color = blue, very thick] (\xtrois,\ytrois) -- (\xtrois,\ymax+0.4);
    \fill [fill=blue, fill opacity=0.1] (\xtrois,\ytrois) rectangle (\xmax,\ymax);

    \node[color = red] at (\xquatre+2.2,\yquatre+1.4) {$\Ncal(p_4)$};
    \draw[color = red, very thick] (\xquatre,\yquatre) -- (\xmax+0.4,\yquatre+0.01);
    \draw[color = red, very thick] (\xquatre,\yquatre) -- (\xquatre,\ymax+0.4);
    \fill [fill=red, fill opacity=0.1] (\xquatre,\yquatre) rectangle (\xmax,\ymax);
    
    \draw[dashed, -] (\xtrois,\ytrois) -- (\xquatre,\yquatre);
    \node at (6.7,6.5) {$\Ncal(p_3,p_4)$};

    \node[color = violet] at (10-0.3,8-0.3) {$q$};
    \node[color = violet] at (10,8) {$\bullet$};
    
    \draw[thick, ->] (0,0) -- (\xmax+0.4,0);
    \draw[thick, ->] (0,0) -- (0,\ymax+0.4);

    \end{tikzpicture}

    \caption{Illustration of Example \ref{Ex_KeyEx}}
    \label{fig:KeyExample}

\end{figure}
\end{ex}

As a consequence of the main theorem, we recover Campillo’s criterion.
\begin{cor}
   Let $m, k_1,\dots, k_r, l_1,\dots l_s\in \Gamma$, such that
\[
m-k_i\in \overline{\Gamma},\text{ }m-l_j\in \overline{\Gamma},\text{ and }\sum_i k_i - \sum_j l_j \in \overline{\Gamma},
\]
for all $i\in \{1,\ldots,r\}$ and $j\in\{1,\ldots,s\}$,then $m + \sum_i k_i - \sum_j l_j \in \Gamma^s$. 
\end{cor}

\begin{proof}
    Since $m\in\Gamma$ and $\sum_i k_i - \sum_j l_j \in \overline{\Gamma}$, we obtain $m+\sum_i k_i - \sum_j l_j \in \overline{\Gamma}$. Now, consider $\{p_1,\ldots,p_n\}$ a set of generators of $\Gamma$. Since $m, k_1,\dots, k_r, l_1,\dots l_s\in \Gamma$, there exist subsets $I_0,I_1,\ldots,I_r,J_1,\ldots,J_s$ of $\{1,\ldots,n\}$ such that $m=\sum_{v\in I_0}c_{v}p_{v}$, $k_i=\sum_{u\in I_i}b_{i,u}p_{u}$ and $l_j=\sum_{t\in J_j}a_{j,t}p_{t}$, where $\{c_v,b_{i,u},a_{j,t}\}\subset\mathbb{Z}_{>0}$. In particular, $m\in\mathcal{N}(p_v)$ for all $v\in I_0$, $k_i\in\mathcal{N}(p_u)$ for all $u\in I_i$ and $l_j\in\mathcal{N}(p_t)$ for all $t\in J_j$.
    Now, since $m-k_i\in \overline{\Gamma}$ and $m-l_j\in \overline{\Gamma}$ for all $i,j$, it follows that $m\in\mathcal{N}(k_i)$ and $m\in\mathcal{N}(l_j)$ for all $i,j$. Thus, $m\in\mathcal{N}(p_i)$ for all $i\in \bigcup_{i=0}^rI_i\cup\bigcup_{j=1}^{s}J_j$.
    Since $\sum_i k_i - \sum_j l_j \in \overline{\Gamma}$, we obtain that $m+\sum_i k_i - \sum_j l_j \in \mathcal{N}(m)$. Therefore, $m+\sum_i k_i - \sum_j l_j \in \mathcal{N}(p_i)$ for all $i\in \bigcup_{i=0}^rI_i\cup\bigcup_{j=1}^{s}J_j$.
    Let $K\subset \{1,\ldots,n\}$ be such that $m+\sum_i k_i - \sum_j l_j \notin \mathcal{N}(p_i\mid i\in K)$. It follows that $(\bigcup_{i=0}^rI_i\cup\bigcup_{j=1}^{s}J_j)\bigcap K= \emptyset$. Thus,
    $$m+\sum_i k_i - \sum_j l_j=\sum_{v\in I_0}c_{v}p_{v}+\sum_i \sum_{u\in I_i}b_{i,u}p_{u} - \sum_j \sum_{t\in J_j}a_{j,t}p_{t}\in \sum_{i\in\bar{I}}p_i\mathbb{Z}\cup\{0\}\subset\sum_{i\notin K}p_i\mathbb{Z}\cup\{0\}.$$
    Hence, $m+\sum_i k_i - \sum_j l_j$ holds $(2)$ in theorem \ref{Th_Main}, obtaining that $m+\sum_i k_i - \sum_j l_j\in\Gamma^s$.
\end{proof}



We know from Proposition \ref{Prop_CampilloSat} that Campillo's saturation is always included in the Lipschitz saturation, even for general algebras, and that the two coincide for complex curves.  We conclude this paper by providing the following example which shows that these two notions do not coincide for complex varieties of higher dimension.\vspace{1cm}

\begin{ex}[\textbf{Campillo saturation and Lipschitz saturation are different}]
    Let $(X,0)$ be a toric singularity such that $\Ocal_{X,0} \simeq \C\{u,v^3,v^4,v^5,u^3v,u^3v^2\} \subset \C\{u,v\} \simeq \Ocal_{\overline{X},0}$. By Theorem \ref{Th_Main}, one can verify that $u^2v^2\in \Ocal_{X^s,0}$. In order to verify that $u^2v^2$ is not in Campillo's saturation, let us consider $p,p_1,\dots, p_n,q_1,\dots,q_m \in \Ocal_{X,0}$ with $p/p_i, p/q_i$ and $\prod_i p_i/\prod_j q_j  \in \C\{u,v\}$ such that $p\times\prod_i p_i/\prod_j q_j = u^2v^2$. Since $\C\{u,v\}$ is a UFD, we get $p=\text{unit}\times u^\alpha v^\beta$ with $\alpha,\beta\in \{0,1,2\}$. However the condition $p\in \Ocal_{X,0}$ imposes $\beta = 0$. But since, for all $i$, we have $p_i \mid p$ and $q_i\mid p$ in $\C\{u,v\}$, then $v \nmid p_i$ and  $v\nmid q_i$. Hence, having $p\times\prod_i p_i/\prod_j q_j = u^2v^2$ is not possible.
\end{ex}

\subsection{A generating set of the semigroup $\Gamma^s$}

The main theorem provides a combinatorial method for computing the semigroup associated with the Lipschitz saturation. The problem now is to determine a generating set for this semigroup, since the conditions given by the theorem are stated pointwise; that is, they must be verified for each monomial individually. The following result provides a partial answer to this problem in the case of smooth normalization.

\begin{rmk}Consider $\mathcal{N} := \cap_{i} \mathcal{N}(p_i)$ and let be $\gamma\in \mathbb{Z}^2 \cap \mathcal{N}$. Then $\gamma\in\Gamma^s$. Indeed: if an element $\gamma\in \mathbb{Z}^2 \cap \mathcal{N}$, this implies that $\gamma\in\mathcal{N}(p_i \mid i\in I)$ for every $I\subseteq \{1,\dots,n\}$. So the condition (2) of the theorem is empty.
\end{rmk}

\begin{rmk}
    As we mention before, by \cite[Theorem 3.2]{DuarteGiles2023ToricLipSat} we know that $\Ocal_{X^s,0} \simeq \C\{\Gamma^s\}$, but only in the case where $\Bar{X}$ is smooth. Let $\Gamma\subset\mathbb{Z}^d$ be a semigroup associated to this variety. With this assumption, we can suppose that $\overline{\Gamma}=\mathbb{N}^d$. In particular, this implies that $\langle e_i\rangle_{\mathbb{R}_\geq0}\cap\Gamma\neq\emptyset$ for all $i\in\{1,\ldots,d\}$, where $\{e_i\}$ are the canonical vectors of $\mathbb{R}^d$.
\end{rmk}

\begin{definition}\label{B}
    Let $\Gamma\subset\mathbb{Z}^d$ be such that $\overline{\Gamma}=\mathbb{N}^d$, and let $\{p_1,\ldots,p_n\}$ be a set of generators of $\Gamma$. For each $i\in\{1,\ldots,d\}$, we define $b_i=\min_{\gamma\in\Gamma\cap\langle e_i\rangle}\{\pi_i(\gamma) \}\in\mathbb{Z}_{>0}$ and $c_i=\max_{j\in\{1,\ldots,n\}}\{\pi_i(p_j)\}$. In addition, we define $\mathcal{B}=\{\gamma\in\mathbb{N}^d\mid \gamma< b\}$, where $b=(b_1+c_1,\ldots,b_d+c_d)\in\mathbb{Z}^d_{>0}$ and $\gamma\leq b$ if and only if $\pi_i(\gamma)<\pi_i(b)$ for all $i\in\{1,\ldots,d\}$.
\end{definition}

\begin{prop}
    Let $\Gamma\subset\mathbb{Z}^d$ be a semigroup such that $\overline{\Gamma}=\mathbb{N}^d$. Then $\Gamma^s\cap\mathcal{B}$ is a finite generator system for $\Gamma^s$, where $\mathcal{B}$ is the set defined in \ref{B}.
\end{prop}
\begin{proof}
    Since $\mathcal{B}\cap\mathbb{Z}^d$ is finite, it is clear that $\Gamma^s\cap\mathcal{B}$ is finite. Let $\gamma\in\Gamma^s$ be such that $\gamma\notin\mathcal{B}$, then $\pi_i(\gamma)\geq \pi_i(b)=b_i+c_i$ (see definition \ref{B}). By the definition of $b_i$, $b_i \cdot e_i\in\Gamma$, where $e_i$ is the i-th canonical vector, and $b_i\cdot e_{i_0}=p_{i_0}$. By the definition of $c_i$, we have $c_i\geq b_i$.

    Let $a_i$ be the natural number such that $b_i+c_i>\pi_i(\gamma-a_i\cdot p_{i_0})\geq c_i$ and let $\gamma'=\gamma-a_i\cdot p_{i_0}$. We are going to show that $\gamma'\in\Gamma^s$. Consider $K\subset\{1,\ldots,n\}$ such that $\gamma'\notin\mathcal{N}(p_j\mid j\in J)$. By Lemma \ref{Lem_NewtonPolCaracScalarProd}, there exists $v\in\mathbb{N}^d$ such that $\langle \gamma',v\rangle<\langle p_j,v\rangle$ for all $j\in\{1,\ldots,n\}$. Since $\pi_i(\gamma')\geq c_i\geq\pi_i(p_j)$ for all $j\in\{1,\ldots,n\}$, $\langle \gamma',v'\rangle<\langle p_j,v'\rangle$ for all $j\in\{1,\ldots,n\}$, where $v'=v-\pi_i(v)\cdot e_i$. Notice that $\pi_j(\gamma')=\pi_j(\gamma)$ for all $j\neq i$ and $\pi_i(v')=0$, this implies that $\langle \gamma,v'\rangle=\langle \gamma',v'\rangle<\langle p_j,v'\rangle$ for all $j\in\{1,\ldots,n\}$. By Lemma \ref{Lem_NewtonPolCaracScalarProd}, we obtain $\gamma\notin \mathcal{N}(p_j\mid j\in J)$. 
    
    By Theorem \ref{Th_Main}, $\gamma=\sum_{j\notin K}m_j\cdot p_j$, for some $\{m_j\}\subset\mathbb{Z}$. Since $b_i\leq c_i\leq \pi_i(\gamma')<\pi_i(\gamma)$, we have $\gamma,\gamma'\in\mathcal{N}(p_{i_0})$. Thus, $i_0\notin K$. Then $\gamma'=\gamma-a_i\cdot p_{i_0}=\sum_{j\notin K}m_j\cdot p_j-a_i\cdot p_{i_0}\in\sum_{i\notin K}\mathbb{Z}p_i\cup\{0\}$, and by Theorem \ref{Th_Main}, $\gamma'\in\Gamma^s$.

    Applying this process in an iterative way for each $\pi_j(\gamma)\geq b_i+c_i$, we obtain a new vector $\bar{\gamma}\in\Gamma^s\cap\mathcal{B}$. Defining $a_j=0$ for all $j\in \{1,\ldots,n\}$ such that $\pi_j(\gamma)\leq b_i+c_i$, we obtain $\gamma=\bar{\gamma}+\sum_{j=1}^{n}a_j\cdot p_{j_0}$. Thus $\gamma\in\langle\Gamma^s\cap\mathcal{B}\rangle_{\mathbb{Z}_{\geq 0}}$.
\end{proof}

The previous proposition provides a bound for determining a generating set of the semigroup; however, in general, such a set need not be minimal.







\bibliographystyle{abbrv}
\bibliography{Biblio.bib}

\vspace{1cm}(F. Bernard) University of Toronto, Department of Mathematics,


\hspace{2cm} \textit{Email address:} francois.bernard@utoronto.ca\\

(E. Chávez-Martínez) Instituto de Ingeniería y Tecnología, UACJ, Ciudad Juárez, México


\hspace{3.3cm} \textit{Email address:} enrique.chavez@uacj.mx\\

(A.E. Giles Flores) Universidad Autónoma de Aguascalientes, Departamento de Matem\'aticas y F\'isica. 


\hspace{2.9cm} \textit{Email address:} arturo.giles@edu.uaa.mx

\end{document}